\documentclass{amsart}
\usepackage{amsthm,amssymb,mathtools}
\usepackage[british]{babel}
\usepackage{enumitem}
\usepackage[latin1]{inputenc}
\usepackage{url}
\usepackage{hyperref}
\usepackage{tikz}
\usetikzlibrary{cd}

\newtheorem{theorem}{Theorem}
\numberwithin{theorem}{section}

\newtheorem{lemma}[theorem]{Lemma}
\newtheorem{proposition}[theorem]{Proposition}
\newtheorem{definition}[theorem]{Definition}

\newcommand{\atrs}{\mathbf{ATR_0^{\operatorname{set}}}}
\newcommand{\prs}{\mathbf{PRS\omega}}
\newcommand{\rca}{\mathbf{RCA_0}}

\newcommand{\supp}{\operatorname{supp}}

\newcommand{\bh}{\operatorname{BH}}

\newcommand{\rng}{\operatorname{rng}}

\newcommand{\lef}{<^{\operatorname{fin}}}

\title[Bachmann-Howard Fixed Points]{A Categorical Construction of Bachmann-Howard Fixed Points\footnotemark[1]}
\author{Anton Freund}

\address{Fachbereich Mathematik, Technische Universit\"at Darmstadt, Schlossgartenstr.~7, 64289~Darmstadt, Germany}
\email{freund@mathematik.tu-darmstadt.de}

\begin{document}

\maketitle
{\let\thefootnote\relax\footnotetext{\copyright~2019. This is the accepted version of the following article: A Categorical Construction of Bachmann-Howard Fixed Points, Bulletin of the London Mathematical Society (in press), which has been published in final form at \href{https://doi.org/10.1112/blms.12285}{doi:10.1112/blms.12285}.}

\begin{abstract}
Peter Aczel has given a categorical construction for fixed points of normal functors, i.e.~dilators which preserve initial segments. For a general dilator $X\mapsto T_X$ we cannot expect to obtain a well-founded fixed point, as the order type of $T_X$ may always exceed the order type of $X$. In the present paper we show how to construct a Bachmann-Howard fixed point of $T$, i.e.~an order $\bh(T)$ with an ``almost'' order preserving collapse $\vartheta:T_{\bh(T)}\rightarrow\bh(T)$. Building on previous work, we show that $\Pi^1_1$-comprehension is equivalent to the assertion that $\bh(T)$ is well-founded for any dilator $T$.
\end{abstract}

\section{Introduction: Set existence and well-foundedness}

The present paper contributes to a research program known as reverse mathematics (see~\cite{simpson09}), which aims to answer the following question: Which infinite sets are indispensable for the proof of a given mathematical theorem?

To make the previous question precise one can consider $\omega$-models, which are defined as collections $\mathcal M\subseteq\mathcal P(\mathbb N)$ of subsets of the natural numbers. For such a model we can ask whether it satisfies certain statements about finite objects (coded by natural numbers) and infinite collections of such objects (represented by subsets of $\mathbb N$). For example, the formula
\begin{equation*}
 \exists_{X\subseteq\mathbb N}\forall_{n\in\mathbb N}(n\in X\leftrightarrow\exists_{m\in\mathbb N}\,m+m=n)
\end{equation*}
holds in an $\omega$-model $\mathcal M$ if, and only if, $\mathcal M$ contains the set of even numbers. More generally, we evaluate a formula in an $\omega$-model $\mathcal M$ by replacing all quantifiers $\exists_{X\subseteq\mathbb N}$ and $\forall_{X\subseteq\mathbb N}$ with $\exists_{X\in\mathcal M}$ and $\forall_{X\in\mathcal M}$, respectively (the quantifiers $\exists_{n\in\mathbb N}$ and $\forall_{n\in\mathbb N}$ remain unchanged).

An $\omega$-model satisfies recursive comprehension if it contains any set that is computable relative to some of its members. Classical results of reverse mathematics show that an $\omega$-model of recursive comprehension does always satisfy the intermediate value theorem, but not necessarily the Bolzano-Weierstra{\ss} theorem (see~\cite[Theorems~II.6.6~and~III.2.2]{simpson09}). In the following we assume that all $\omega$-models satisfy recursive comprehension.

It turns out that the Bolzano-Weierstra{\ss} theorem corresponds to a stonger comprehension principle: Roughly speaking, a formula is called arithmetical if it does not involve quantification over subsets of $\mathbb N$ (e.\,g.~the above formula is not arithmetical, due to the quantifier $\exists_{X\subseteq\mathbb N}$, but the subformula $\exists_{m\in\mathbb N}\,m+m=n$ is). An $\omega$-model $\mathcal M$ satisfies arithmetical comprehension if it contains any set that can be defined by an arithmetical formula with parameters from $\mathcal M$. Note that arithmetical comprehension implies recursive comprehension.  The classical result that we have cited above does actually show the following stronger assertion: An $\omega$-model satisfies the Bolzano-Weierstra{\ss} theorem if, and only if, it satisfies arithmetical comprehension (in fact this holds for a more general class of models, but $\omega$-models are particularly intuitive). Hence arithmatically defined sets are indispensable for this theorem.

Our goal is to relate set existence to the notion of well-foundedness, which plays an important role throughout mathematics. The idea is to consider transformations between well-orders, which are known as (type-one) well-ordering principles. This approach was first realized by Girard~\cite{girard87}, who considered the transformation of an order $X=(X,\leq_X)$ into the set
\begin{equation*}
\omega^X=\{\langle x_0,\dots,x_{n-1}\rangle\,|\,x_{n-1}\leq_X\dots\leq_X x_0\}
\end{equation*}
of finite descending sequences with entries from $X$, ordered lexicographically (think of Cantor normal forms). Girard has shown that an $\omega$-model satisfies arithmetical comprehension if, and only if, it satisfies the statement that ``$\omega^X$ is well-founded for any well-order~$X$". Let us point out that the order $\omega^X$ from the previous paragraph is computable relative to~$X$, so that $X\in\mathcal M$ implies $\omega^X\in\mathcal M$. In other words, the strength of the well-ordering principle lies in the preservation of well-foundedness, not in the existence of the set $\omega^X$.

Many set existence principles have been characterized in terms of type-one well-ordering principles (see~\cite{freund-equivalence} for a list of references). However, such a characterization cannot be given for one particularly important set existence principle: An $\omega$-model $\mathcal M$ satisfies $\Pi^1_1$-comprehension if it contains any set of the form
\begin{equation}\label{eq:Pi11-comp}
\{n\in\mathbb N\,|\,\forall_{X\in\mathcal M}\,\varphi(n,X)\},
\end{equation}
where $\varphi$ is an arithmetical formula with parameters from $\mathcal M$. The point is that any type-one well-ordering principle can be expressed by a formula $\forall_{X\subseteq\mathbb N}\exists_{Y\subseteq\mathbb N}\,\varphi$, where $\varphi$ is arithmetical. It is known that $\Pi^1_1$-comprehension cannot be equivalent to a formula of this form.

Why is $\Pi^1_1$-comprehension interesting? Firstly, $\Pi^1_1$-comprehension is equivalent to important mathematical theorems, such as the Cantor-Bendixson theorem or the result that any countable Abelian group can be constructed as the direct sum of a divisible and a reduced group (see~\cite[Theorems~VI.1.6~and~VI.4.1]{simpson09}). Secondly, there is a fundamental difference between $\Pi^1_1$-comprehension and weaker principles: An $\omega$-model
\begin{equation*}
 \mathcal M=\{\mathcal M_0,\mathcal M_1,\mathcal M_2,\dots\}
\end{equation*}
of arithmetical comprehension can be built ``from below'', by setting $\mathcal M_{k+1}=\{n\in\mathbb N\,|\,\varphi_k(n)\}$ for some systematic enumeration $\varphi_0,\varphi_1,\dots$ of the arithmetical formulas, where $\varphi_k$ may contain the parameters $\mathcal M_0,\dots,\mathcal M_k$. An $\omega$-model of $\Pi^1_1$-comprehension cannot be built in the same way, since the quantifier $\forall_{X\in\mathcal M}$ in (\ref{eq:Pi11-comp}) changes its range whenever we extend $\mathcal M$ by a set~$\mathcal M_{k+1}$, so that the construction of $\mathcal M_0,\dots,\mathcal M_k$ may be invalidated at this point. One can describe this difference by saying that arithmetical comprehension is predicative, while $\Pi^1_1$-comprehension is impredicative (see~\cite{feferman05} for a deeper discussion of these notions).

Rathjen~\cite{rathjen-wops-chicago,rathjen-atr} and Montalb\'an~\cite{montalban-draft,montalban-open-problems} have conjectured that $\Pi^1_1$-comprehension can be characterized by a type-two well-ordering principle. Such a principle should take a type-one well-ordering principle $X\mapsto T_X$ as input and yield a well-order as output. In~\cite{freund-equivalence} (which is based on the author's PhD thesis~\cite{freund-thesis} and an earlier arXiv preprint~\cite{freund-bh-preprint}) we have made a significant step towards this conjecture. The basic idea was that the output should be a certain type of fixed-point of the input $X\mapsto T_X$.

What is the correct notion of fixed point? Let us first observe that we cannot expect to obtain well-founded fixed points in the obvious sense: Consider the transformation 
\begin{equation*}
 X\mapsto T_X=X\cup\{\top\}
\end{equation*}
that extends a given order~$X$ by a new maximal element $\top$. Let $X$ be an arbitrary well-order. Then there is a unique ordinal~$\alpha$ with $X\cong\alpha$. Clearly we have $T_X\cong\alpha+1$. By the uniqueness of $\alpha$ we can conclude
\begin{equation*}
 T_X\not\cong X.
\end{equation*}
In~\cite{freund-equivalence} we have formulated precise conditions under which a function
\begin{equation*}
 \vartheta:T_X\rightarrow X
\end{equation*}
can be considered as ``almost'' order preserving. If such a function exists, then $X$ is called a Bachmann-Howard fixed point of the transformation $T$. In fact, these notions are only defined if $T$ is a particularly uniform transformation of orders, namely a dilator in the sense of Girard~\cite{girard-pi2}. Full details of all relevant definitions will be recalled in the following section. Relying on these notions, we can now state the main result of~\cite{freund-equivalence}:
\begin{theorem}\label{thm:abstract-bhp}
 The following are equivalent over the theory $\atrs$:
 \begin{enumerate}[label=(\roman*)]
  \item The principle of $\Pi^1_1$-comprehension.
  \item The statement that every dilator has a well-founded Bachmann-Howard fixed point.
 \end{enumerate}
\end{theorem}
Here $\atrs$ is Simpson's~\cite{simpson82} set-theoretic version of arithmetical transfinite recursion (see the following section for details). The point is that $\atrs$ is a predicative theory (in the broad sense). In this crucial respect the theory in the ``background'' is simpler than the principle of $\Pi^1_1$-comprehension, which is characterized by the theorem.

So what is new in the present paper? We will present an explicit construction, which transforms a given dilator $T$ into a Bachmann-Howard fixed point $\bh(T)$ of $T$ (see Theorem~\ref{thm:bhT-fixed point}). More precisely, the order $\bh(T)$ can be constructed by a primitive recursive set function, as explained at the beginning of the following section. We will also show that $\bh(T)$ is minimal, in the sense that it can be embedded into any other Bachmann-Howard fixed point of~$T$ (see Theorem~\ref{thm:bhT-minimal}). It follows that $\bh(T)$ is well-founded if, and only if, the dilator $T$ has some well-founded Bachmann-Howard fixed point. This improves Theorem~\ref{thm:abstract-bhp} as follows:
\begin{theorem}\label{thm:pred-bhp}
 The following are equivalent over the theory $\atrs$:
 \begin{enumerate}[label=(\roman*)]
  \item The principle of $\Pi^1_1$-comprehension.
  \item The statement that the order $\bh(T)$ is well-founded for every dilator~$T$.
 \end{enumerate}
\end{theorem}
The crucial point is that $\bh(T)$ is constructed ``from below'', i.\,e.~in a predicative way. Hence the impredicative principle of $\Pi^1_1$-comprehension is reduced to a predicative construction and a statement about the preservation of well-foundedness, over a predicative background theory. Statement~(ii) from Theorem~\ref{thm:abstract-bhp} will be called the abstract Bachmann-Howard principle, since it asserts the existence of a fixed point without constructing one. Statement~(ii) of Theorem~\ref{thm:pred-bhp} will be called the predicative Bachmann-Howard principle.

To conclude this introduction we explain how to construct the Bachmann-Howard fixed point~$\bh(T)$. Our starting point is a construction due to Aczel~\cite{aczel-phd,aczel-normal-functors}: Given an endofunctor~$T$ on the category of linear orders, let $X$ be the direct limit of the diagram
\begin{equation*}
 X_0:=\emptyset\xrightarrow{\mathmakebox[1cm]{\iota_0}}X_1:=T_{X_0}\xrightarrow{\mathmakebox[1cm]{\iota_1:=T_{\iota_0}}}X_2:=T_{X_1}\xrightarrow{\mathmakebox[1cm]{}}\cdots.
\end{equation*}
If $T$ preserves direct limits, then we have $T_X\cong X$. Aczel has shown that $X$ is well-founded for an important class of functors, which correspond to normal functions on the ordinals (the crucial condition requires that the range of $T_f:T_X\rightarrow T_Y$ is an initial segment of the order $T_Y$ whenever the range of $f:X\rightarrow Y$ is an initial segment of $Y$).

The transformation $X\mapsto T_X=X\cup\{\top\}$ that we have considered above is readily extended into a functor (set $T_f(\top)=\top$). We have seen that $T_X\cong X$ cannot hold for any well-order~$X$. Indeed, the fixed point that results from Aczel's construction is isomorphic to the negative integers. In order to obtain well-founded Bachmann-Howard fixed points we will modify Aczel's construction as follows: Given a linear order $X$, one can define an order $\vartheta_T(X)$ with an ``almost'' order preserving collapse $\vartheta_X:T_X\rightarrow\vartheta_T(X)$. In fact the situation is somewhat more complicated: To define the order relation on $\vartheta_T(X)$ we already need a function $\iota_X:X\rightarrow\vartheta_T(X)$ between the underlying sets. After the order has been defined we will want $\iota_X$ to be an order embedding. We will introduce a notion of (good) Bachmann-Howard system to keep track of this requirement. Relying on this notion, we will construct a diagram of the form
\begin{equation*}
 \begin{tikzcd}
  & T_{X_0}\arrow{d}{\vartheta_{X_0}}\arrow{r}{T_{\iota_{X_0}}} & T_{X_1}\arrow{d}{\vartheta_{X_1}}\arrow[r] & \cdots\\
  X_0:=\emptyset\arrow{r}{\iota_{X_0}} & X_1:=\vartheta_T(X_0)\arrow{r}{\iota_{X_1}} & X_2:=\vartheta_T(X_1)\arrow[r] & \cdots.
 \end{tikzcd}
\end{equation*}
The order $\bh(T)$ will be defined as the direct limit of the orders~$X_n$. To establish that $\bh(T)$ is a Bachmann-Howard fixed point of~$T$ we will show that the functions $\vartheta_{X_n}:T_{X_n}\rightarrow X_{n+1}$ glue to an almost order preserving function $\vartheta:T_{\bh(T)}\rightarrow\bh(T)$.

Finally, let us mention that it is possible to construct Bachmann-Howard fixed points in an even weaker setting: In~\cite{freund-computable} we show that $\bh(T)$ can be represented by an ordinal notation system~$\vartheta(T)$, which is computable relative to (a suitable representation of) $T$. The assertion that $\vartheta(T)$ is well-founded for any dilator $T$ will be called the computable Bachmann-Howard principle. According to~\cite[Theorem~4.6]{freund-computable} this principle is still equivalent to $\Pi^1_1$-comprehension, even over the base theory $\rca$. The construction in~\cite{freund-computable} is considerably more technical than the one in the present paper, not least because it depends on a representation of dilators in second-order arithmetic. It is often pointed out that ordinal notation systems such as $\vartheta(T)$ are difficult to understand from a purely syntactical standpoint. The present paper provides a transparent semantical construction of $\bh(T)$, which is crucial for the understanding of Bachmann-Howard fixed points.

The author would like to point out that parts of this paper are based on Sections~2.2 and~2.4 of his PhD thesis~\cite{freund-thesis}.

\section{Preliminaries: Dilators and collapsing functions}

In the present section we give precise definitions of notions that have been mentioned in the introduction. We will be most interested in dilators and their Bachmann-Howard fixed points. At the end of the section we discuss the meta theory in which the present paper is supposed to be formalized.

Dilators are particularly uniform endofunctors on the category of linear orders, with order embeddings as morphisms. In order to express the uniformity condition we consider the finite subset functor on the category of sets, which is given by
\begin{align*}
 [X]^{<\omega}&=\text{``the set of finite subsets of $X$''},\\
 [f]^{<\omega}(a)&=\{f(x)\,|\,x\in a\}.
\end{align*}
We will also apply $[\cdot]^{<\omega}$ to linear orders, omitting the forgetful functor to their underlying sets. Conversely, a subset of a linear order will often be considered as a suborder. The following notion is essentially due to Girard~\cite{girard-pi2}:

\begin{definition}\label{def:prae-dilator}
 A prae-dilator consists of
 \begin{enumerate}[label=(\roman*)]
  \item an endofunctor $X\mapsto T_X$ of linear orders and
  \item a natural transformation $\supp^T:T\Rightarrow[\cdot]^{<\omega}$ that computes supports, in the following sense: For any linear order $X$ and any element $\sigma\in T_X$ we have $\sigma\in\rng(T_{\iota_\sigma})$, where $\iota_\sigma:\supp^T_X(\sigma)\hookrightarrow X$ is the inclusion.
 \end{enumerate}
 If $T_X$ is well-founded for any well-order $X$, then $(T,\supp^T)$ is called a dilator.
\end{definition}

We point out that our notion of prae-dilator is slightly different from Girard's notion of pre-dilator, which involves an additional monotonicity condition. The latter is automatic for well-orders, so that the difference vanishes in the case of dilators. Also, Girard's definition does not involve the natural transformation $\supp^T$. Instead, it demands that $T$ preserves direct limits and pull-backs. It is not hard to see that the two definitions are equivalent (see~\cite[Remark~2.2.2]{freund-thesis}). Nevertheless it will be very useful to make the supports explicit.

To say when a function is ``almost'' order preserving we need the following notation: Given a linear order $(X,<_X)$, we define a preorder $\lef_X$ on $[X]^{<\omega}$ by stipulating
\begin{equation*}
 a\lef_X b\quad:\Leftrightarrow\quad\text{``for any $s\in a$ there is a $t\in b$ with $s<_X t$''.}
\end{equation*}
We will write $s\lef_X b$ and $a\lef_X t$ rather than $\{s\}\lef_X b$ resp.~$a\lef_X \{t\}$ for singletons. The relation $\leq^{\operatorname{fin}}_X$ is defined in the same way. The following notion was introduced in~\cite{freund-equivalence}, based on the author's PhD thesis~\cite{freund-thesis} and an earlier arXiv preprint~\cite{freund-bh-preprint}. It is inspired by the definition of the Bachmann-Howard ordinal, in particular by the variant due to Rathjen (cf.~\cite[Section~1]{rathjen-weiermann-kruskal}).

\begin{definition}\label{def:bachmann-howard-collapse}
 Consider a prae-dilator $(T,\supp^T)$ and an order $X$. A function
 \begin{equation*}
  \vartheta:T_X\rightarrow X
 \end{equation*}
 is called a Bachmann-Howard collapse if the following holds for all $\sigma,\tau\in T_X$:
 \begin{enumerate}[label=(\roman*)]
  \item If we have $\sigma<_{T_X}\tau$ and $\supp^T_X(\sigma)\lef_X\vartheta(\tau)$, then we have $\vartheta(\sigma)<_X\vartheta(t)$.
  \item We have $\supp^T_X(\sigma)\lef_X\vartheta(\sigma)$.
 \end{enumerate}
 If such a function exists, then $X$ is called a Bachmann-Howard fixed point of $T$.
\end{definition}

We can now officially introduce the following principle:

\begin{definition}\label{def:abstract-bhp}
 The abstract Bachmann-Howard principle is the assertion that every dilator has a well-founded Bachmann-Howard fixed point.
\end{definition}

To give an example we recall the functor
\begin{equation*}
 X\mapsto T_X=X\cup\{\top\}
\end{equation*}
that was considered in the introduction. We obtain a dilator if we set $\supp^T_X(\top)=\emptyset$ and $\supp^T_X(x)=\{x\}$ for $x\in X\subseteq T_X$. It is straightforward to check that the function
\begin{equation*}
 \vartheta:T_\omega\rightarrow\omega\quad\text{with}\quad\begin{cases}
                                                          \vartheta(\top)=0,\\
                                                          \vartheta(n)=n+1,
                                                         \end{cases}
\end{equation*}
is a Bachmann-Howard collapse. Conversely, if $\vartheta:T_X\rightarrow X$ is any Bachmann-Howard collapse, then we can define an embedding $f:\omega\rightarrow X$ by setting $f(0)=\vartheta(\top)$ and $f(n+1)=\vartheta(f(n))$.

Let us now discuss the formalization of the previous notions: The meta theory of the present paper is primitive recursive set theory with infinity ($\prs$), as introduced by Rathjen~\cite{rathjen-set-functions} (see also the detailed exposition in \cite[Chapter~1]{freund-thesis}). This theory has a function symbol for each primitive recursive set function in the sense of Jensen and Karp~\cite{jensen-karp}. When we speak about class-sized objects of a certain kind (e.g.~about arbitrary endofunctors on linear orders) we need to observe two restrictions: Firstly, we will only consider class-sized objects which are primitive recursive. Secondly, we cannot quantify over all primitive recursive set functions. However, we can quantify over a primitive recursive family of class-sized functions, by quantifying over its set-sized parameters. Statements about class-sized objects should thus be read as schemata.

In our context the restrictions from the previous paragraph are harmless: Girard~\cite{girard-pi2} has shown that \mbox{(prae-)dilators} are essentially determined by their restrictions to the category of natural numbers. In~\cite[Section~2]{freund-computable} we deduce that any prae-dilator is naturally equivalent to one that is given by a primitive recursive set function. Indeed we show that there is a single primitive recursive family that comprises (isomorphic copies of) all prae-dilators. Thus the abstract Bachmann-Howard principle can be expressed by a single sentence in the language of $\prs$. One can even represent prae-dilators in second-order arithmetic (see again \cite[Section~2]{freund-computable}), but this will not be relevant for the present paper.

Let us also point out what these methodological remarks mean for the construction of Bachmann-Howard fixed points: In the introduction we have said that the transformation of a (prae-)dilator $T$ into its minimal Bachmann-Howard fixed point $\bh(T)$ can be achieved by a primitive recursive set function. We can now be more precise about this claim: If $(T^u)_{u\in\mathbb V}$ is a primitive recursive family of prae-dilators, indexed by elements of the set-theoretic universe, then the function $u\mapsto\bh(T^u)$ will be primitive recursive as well.

Finally, we recall Simpson's~\cite{simpson82,simpson09} set-theoretic version $\atrs$ of arithmetical transfinite recursion: It results from $\prs$ by adding axiom beta (which asserts that every well-founded relation can be collapsed to the $\in$-relation) and the axiom of countability (which asserts that every set is countable). The additional axioms of $\atrs$ are needed for Theorem~\ref{thm:pred-bhp}, since they are used in the proof of Theorem~\ref{thm:abstract-bhp} in~\cite{freund-equivalence}.

\section{Bachmann-Howard systems}

In the introduction we have mentioned linear orders $(\vartheta_T(X),<_{\vartheta_T(X)})$ that allow for an ``almost'' order preserving collapse $\vartheta_X:T_X\rightarrow\vartheta_T(X)$. The construction of these orders proceeds in two steps. First, we must define the underlying sets:

\begin{definition}
 Consider a prae-dilator $T$. For each linear order $X$ we define $\vartheta_T(X)$ as the set of terms $\vartheta\sigma$ with $\sigma\in T_X$.
\end{definition}

In view of Definition~\ref{def:bachmann-howard-collapse} the relation $\vartheta\sigma<_{\vartheta_T(X)}\vartheta\tau$ should depend on a comparison between $\supp^T_X(\sigma)$ and $\vartheta\tau$. This is not completely straightforward, because $\supp^T_X(\sigma)$ is a subset of $X$ rather than $\vartheta_T(X)$. To resolve this problem we introduce the following notion:

\begin{definition}\label{def:collapse-bh-system}
 Consider a linear order $X$ together with functions $\iota_X:X\rightarrow\vartheta_T(X)$ and $L_X:X\rightarrow\omega$. Define $L_{\vartheta_T(X)}:\vartheta_T(X)\rightarrow\omega$ by
 \begin{equation*}
  L_{\vartheta_T(X)}(\vartheta\sigma):=\max\{L_X(x)\,|\,x\in\supp^T_X(\sigma)\}+1.
 \end{equation*}
 If we have
 \begin{equation*}
  L_{\vartheta_T(X)}\circ\iota_X=L_X,
 \end{equation*}
 then the tuple $(X,\iota_X,L_X)$ is called a Bachmann-Howard system (for $T$).
\end{definition}

Note that one obtains $L_{\vartheta_T(X)}(\iota_X(x))=L_X(x)<L_{\vartheta_T(X)}(\vartheta\sigma)$ for $x\in\supp^T_X(\sigma)$. This allows for the following recursion:

\begin{definition}\label{def:collapse-bh-system-inequality}
 Let $(X,\iota_X,L_X)$ be a Bachmann-Howard system for the prae-dilator $T$. Relying on recursion over $L_{\vartheta_T(X)}(\vartheta\sigma)+L_{\vartheta_T(X)}(\vartheta\tau)$, we stipulate that $\vartheta\sigma<_{\vartheta_T(X)}\vartheta\tau$ holds precisely if one of the following clauses is satisfied:
 \begin{enumerate}[label=(\roman*)]
  \item We have $\sigma<_{T_X}\tau$ and $[\iota_X]^{<\omega}(\supp^T_X(\sigma))\lef_{\vartheta_T(X)}\vartheta\tau$.
  \item We have $\tau<_{T_X}\sigma$ and $\vartheta\sigma\leq^{\operatorname{fin}}_{\vartheta_T(X)}[\iota_X]^{<\omega}(\supp^T_X(\tau))$.
 \end{enumerate}
\end{definition}

Let us establish the following basic property:

\begin{lemma}\label{lem:collapse-linear}
 If $(X,\iota_X,L_X)$ is a Bachmann-Howard system, then $(\vartheta_T(X),<_{\vartheta_T(X)})$ is a linear order.
\end{lemma}
\begin{proof}
 The antisymmetry of $<_{\vartheta_T(X)}$ follows easily from the antisymmetry of $<_{T_X}$. Trichotomy for $\vartheta\sigma$ and $\vartheta\tau$ is established by induction on $L_{\vartheta_T(X)}(\vartheta\sigma)+L_{\vartheta_T(X)}(\vartheta\tau)$: By symmetry we may assume $\sigma<_{T_X}\tau$. For an arbitrary $x\in\supp^T_X(\sigma)$ the induction hypothesis provides $\iota_X(x)<_{\vartheta_T(X)}\vartheta\tau$ or $\vartheta\tau\leq_{\vartheta_T(X)}\iota_X(x)$. If the former holds for all $x\in\supp^T_X(\sigma)$, we have $\vartheta\sigma<\vartheta\tau$ by clause~(i) of the previous definition. If we have $\vartheta\tau\leq_{\vartheta_T(X)}\iota_X(x)$ for some $x\in\supp^T_X(\sigma)$, we get $\vartheta\tau<_{\vartheta_T(X)}\vartheta\sigma$ by clause~(ii). Finally, we argue by induction on $L_{\vartheta_T(X)}(\vartheta\rho)+L_{\vartheta_T(X)}(\vartheta\sigma)+L_{\vartheta_T(X)}(\vartheta\tau)$ to show that $\vartheta\rho<_{\vartheta_T(X)}\vartheta\sigma<_{\vartheta_T(X)}\vartheta\tau$ implies $\vartheta\rho<_{\vartheta_T(X)}\vartheta\tau$. If $\vartheta\rho<_{\vartheta_T(X)}\vartheta\sigma$ and $\vartheta\sigma<_{\vartheta_T(X)}\vartheta\tau$ hold by the same clause of the previous definition, then it is easy to conclude by induction hypothesis. Now assume that $\vartheta\rho<_{\vartheta_T(X)}\vartheta\sigma$ holds by clause~(i) while $\vartheta\sigma<_{\vartheta_T(X)}\vartheta\tau$ holds by clause~(ii). This means that we have
 \begin{align*}
  \rho&<_{T_X}\sigma &&\text{and} & [\iota_X]^{<\omega}(\supp^T_X(\rho))&\lef_{\vartheta_T(X)}\vartheta\sigma,\\
  \tau&<_{T_X}\sigma &&\text{and} & \vartheta\sigma&\leq^{\operatorname{fin}}_{\vartheta_T(X)}[\iota_X]^{<\omega}(\supp^T_X(\tau)).
 \end{align*}
 If we have $\rho<_{T_X}\tau$ or $\tau<_{T_X}\rho$, then we can conclude by induction hypothesis. It remains to exclude the case $\rho=\tau$: By the assumption $\vartheta\sigma\leq^{\operatorname{fin}}_{\vartheta_T(X)}[\iota_X]^{<\omega}(\supp^T_X(\tau))$, pick an element $x\in\supp^T_X(\tau)=\supp^T_X(\rho)$ with $\vartheta\sigma\leq_{\vartheta_T(X)}\iota_X(x)$. In view of $[\iota_X]^{<\omega}(\supp^T_X(\rho))\lef_{\vartheta_T(X)}\vartheta\sigma$ we also have $\iota_X(x)<_{\vartheta_T(X)}\vartheta\sigma$. The induction hypothesis allows us to conclude $\iota_X(x)<_{\vartheta_T(X)}\iota_X(x)$ by transitivity. This contradicts antisymmetry, as desired. Finally, assume that
 $\vartheta\rho<_{\vartheta_T(X)}\vartheta\sigma$ holds by clause~(ii) while $\vartheta\sigma<_{\vartheta_T(X)}\vartheta\tau$ holds by clause~(i). This means that we have
 \begin{align*}
 \sigma&<_{T_X}\rho &&\text{and} & \vartheta\rho&\leq^{\operatorname{fin}}_{\vartheta_T(X)}[\iota_X]^{<\omega}(\supp^T_X(\sigma)),\\
 \sigma&<_{T_X}\tau &&\text{and} & [\iota_X]^{<\omega}(\supp^T_X(\sigma))&\lef_{\vartheta_T(X)}\vartheta\tau.
 \end{align*}
 The assumption $\vartheta\rho\leq^{\operatorname{fin}}_{\vartheta_T(X)}[\iota_X]^{<\omega}(\supp^T_X(\sigma))$ yields an $x\in\supp^T_X(\sigma)$ with $\vartheta\rho\leq_{\vartheta_T(X)}\iota_X(x)$. By $[\iota_X]^{<\omega}(\supp^T_X(\sigma))\lef_{\vartheta_T(X)}\vartheta\tau$ we also have $\iota_X(x)<_{\vartheta_T(X)}\vartheta\tau$. Using the induction hypothesis we can conclude $\vartheta\rho<_{\vartheta_T(X)}\vartheta\tau$, as required.
\end{proof}

We can now define the collapsing functions mentioned in the introduction:

\begin{definition}\label{def:collapse-approximations}
 Let $(X,\iota_X,L_X)$ be a Bachmann-Howard system for the prae-dilator~$T$. We define a function $\vartheta_X:T_X\rightarrow\vartheta_T(X)$ by setting $\vartheta_X(\sigma)=\vartheta\sigma$.
\end{definition}

Let us recover the conditions from Definition~\ref{def:bachmann-howard-collapse}:

\begin{proposition}\label{prop:collapses-admissible}
 Assume that $(X,\iota_X,L_X)$ is a Bachmann-Howard system for the prae-dilator~$T$. Then the following holds for all $\sigma,\tau\in T_X$:
 \begin{enumerate}[label=(\roman*)]
  \item If we have $\sigma<_{T_X}\tau$ and $[\iota_X]^{<\omega}(\supp^T_X(\sigma))\lef_{\vartheta_T(X)}\vartheta_X(\tau)$, then $\vartheta_X(\sigma)<_{\vartheta_T(X)}\vartheta_X(\tau)$ holds.
  \item We have $[\iota_X]^{<\omega}(\supp^T_X(\sigma))\lef_{\vartheta_T(X)}\vartheta_X(\sigma)$.
 \end{enumerate}
\end{proposition}
\begin{proof}
 Claim~(i) is immediate by the definitions. To establish claim~(ii) we consider the auxiliary function $E_X:\vartheta_T(X)\rightarrow[\vartheta_T(X)]^{<\omega}$ with
 \begin{equation*}
  E_X(\vartheta\sigma):=\{\vartheta\sigma\}\cup\bigcup\{E_X(\iota_X(x))\,|\,x\in\supp^T_X(\sigma)\},
 \end{equation*}
 which is defined by recursion on $L_{\vartheta_T(X)}(\vartheta\sigma)>L_{\vartheta_T(X)}(\iota_X(x))$. One may think of $\vartheta\rho\in E_X(\vartheta\sigma)$ as a subterm of $\vartheta\sigma$: A straightforward induction on $L_{\vartheta_T(X)}(\vartheta\sigma)$ shows that $\vartheta\rho\in E_X(\vartheta\sigma)$ implies $E_X(\vartheta\rho)\subseteq E_X(\vartheta\sigma)$ and $L_{\vartheta_T(X)}(\vartheta\rho)\leq L_{\vartheta_T(X)}(\vartheta\sigma)$. The crucial step towards claim~(ii) is the implication
 \begin{equation*}
  \vartheta\rho\in E_X(\vartheta\sigma)\quad\Rightarrow\quad\vartheta\rho\leq_{\vartheta_T(X)}\vartheta\sigma,
 \end{equation*}
 which we prove by induction on $L_{\vartheta_T(X)}(\vartheta\rho)+L_{\vartheta_T(X)}(\vartheta\sigma)$. Let us distinguish three cases: If we have $\rho=\sigma$, then the claim is immediate. Now assume $\rho<_{T_X}\sigma$. To infer $\vartheta\rho<_{\vartheta_T(X)}\vartheta\sigma$ we must establish $[\iota_X]^{<\omega}(\supp^T_X(\rho))\lef_{\vartheta_T(X)}\vartheta\sigma$. For an arbitrary $x\in\supp^T_X(\rho)$ we have
 \begin{equation*}
  L_{\vartheta_T(X)}(\iota_X(x))<L_{\vartheta_T(X)}(\vartheta\rho)\leq L_{\vartheta_T(X)}(\vartheta\sigma),
 \end{equation*}
as well as
\begin{equation*}
 \iota_X(x)\in E_X(\iota_X(x))\subseteq E_X(\vartheta\rho)\subseteq E_X(\vartheta\sigma).
\end{equation*}
So the induction hypothesis yields $\iota_X(x)\leq_{\vartheta_T(X)}\vartheta\sigma$. Also note that $\iota_X(x)$ and $\vartheta\sigma$ cannot be the same term, since we have $L_{\vartheta_T(X)}(\iota_X(x))<L_{\vartheta_T(X)}(\vartheta\sigma)$. Thus we get $\iota_X(x)<_{\vartheta_T(X)}\vartheta\sigma$, as required. Finally we consider the case $\sigma<_{T_X}\rho$. By the definition of $E_X(\vartheta\sigma)$ we may pick a $y\in\supp^T_X(\sigma)$ with $\vartheta\rho\in E_X(\iota_X(y))$. The induction hypothesis provides $\vartheta\rho\leq_{\vartheta_T(X)}\iota_X(y)$ and thus
\begin{equation*}
 \vartheta\rho\leq^{\operatorname{fin}}_{\vartheta_T(X)}[\iota_X]^{<\omega}(\supp^T_X(\sigma)).
\end{equation*}
 Then we can conclude $\vartheta\rho<_{\vartheta_T(X)}\vartheta\sigma$ by definition. To deduce claim~(ii) of the proposition we observe that $x\in\supp^T_X(\sigma)$ yields
 \begin{equation*}
  \iota_X(x)\in E_X(\iota_X(x))\subseteq E_X(\vartheta\sigma).
 \end{equation*}
 We have just shown that this implies $\iota_X(x)\leq_{\vartheta_T(X)}\vartheta\sigma$. In view of $L_{\vartheta_T(X)}(\iota_X(x))<L_{\vartheta_T(X)}(\vartheta\sigma)$ the terms $\iota_X(x)$ and $\vartheta\sigma$ must be different. So indeed we get $\iota_X(x)<_{\vartheta_T(X)}\vartheta\sigma=\vartheta_X(\sigma)$. 
\end{proof}

In particular we have shown that the condition $\tau<_{T_X}\sigma$ in clause~(ii) of Definition~\ref{def:collapse-bh-system-inequality} becomes redundant: The implication
\begin{equation*}
 \vartheta\sigma\leq^{\operatorname{fin}}_{\vartheta_T(X)}[\iota_X]^{<\omega}(\supp^T_X(\tau))\quad\Rightarrow\quad\vartheta\sigma<_{\vartheta_T(X)}\vartheta\tau
\end{equation*}
follows from $[\iota_X]^{<\omega}(\supp^T_X(\tau))\lef_{\vartheta_T(X)}\vartheta_X(\tau)$ and transitivity. To define the linear order $<_{\vartheta_T(X)}$ we have relied on a function $\iota_X:X\rightarrow\vartheta_T(X)$ which respects the length assignments $L_X$ and $L_{\vartheta_T(X)}$. Now that we have an order on $\vartheta_T(X)$ we want $\iota_X$ to respect it as well:

\begin{definition}
 A Bachmann-Howard system $(X,\iota_X,L_X)$ is called good if the function $\iota_X:X\rightarrow\vartheta_T(X)$ is an order embedding.
\end{definition}

Note that this justifies the arrow $T_{\iota_{X_0}}$ in the second diagram from the introduction: If $\iota_{X_0}$ is an embedding of $X_0$ into $X_1=\vartheta_T(X_0)$, then $T_{\iota_{X_0}}$ is an embedding of $T_{X_0}$ into $T_{X_1}$. Based on this arrow we can also construct the arrow $\iota_{X_1}:X_1\rightarrow X_2$:

\begin{definition}
 Let $(X,\iota_X,L_X)$ be a good Bachmann-Howard system. We define a function $\iota_{\vartheta_T(X)}:\vartheta_T(X)\rightarrow\vartheta_T(\vartheta_T(X))$ by setting $\iota_{\vartheta_T(X)}(\vartheta\sigma):=\vartheta\, T_{\iota_X}(\sigma)$.
\end{definition}

Note that the empty order, together with the unique functions $\iota_\emptyset:\emptyset\rightarrow \vartheta_T(\emptyset)$ and $L_\emptyset:\emptyset\rightarrow\omega$, is a good Bachmann-Howard system for any prae-dilator. Once we have a starting point we can use the following result to construct iterations:

\begin{theorem}\label{thm:bh-systems-iterate}
 Consider a prae-dilator $T$. If $(X,\iota_X,L_X)$ is a good Bachmann-Howard system for $T$, then so is $(\vartheta_T(X),\iota_{\vartheta_T(X)},L_{\vartheta_T(X)})$.
\end{theorem}
\begin{proof}
 Abbreviate $\vartheta^2_T(X):=\vartheta_T(\vartheta_T(X))$ and consider $L_{\vartheta^2_T(X)}:\vartheta^2_T(X)\rightarrow\omega$ as constructed in Definition~\ref{def:collapse-bh-system}. Since $\supp^T$ is a natural transformation and $X$ is a Bachmann-Howard system we can compute
 \begin{alignat*}{3}
  L_{\vartheta^2_T(X)}\circ&\iota_{\vartheta_T(X)}(\vartheta\sigma)=L_{\vartheta^2_T(X)}(\vartheta\, T_{\iota_X}(\sigma))&&=\\
  &=\max\{L_{\vartheta_T(X)}(s)\,|\,s\in\supp^T_{\vartheta_T(X)}(T_{\iota_X}(\sigma))\}+1&&=\\
  &=\max\{L_{\vartheta_T(X)}(s)\,|\,s\in[\iota_X]^{<\omega}(\supp^T_X(\sigma))\}+1&&=\\
  &=\max\{L_{\vartheta_T(X)}(\iota_X(x))\,|\,x\in\supp^T_X(\sigma)\}+1&&=\\
  &=\max\{L_X(x)\,|\,x\in\supp^T_X(\sigma)\}+1&&=L_{\vartheta_T(X)}(\vartheta\sigma).
 \end{alignat*}
 This shows that $(\vartheta_T(X),\iota_{\vartheta_T(X)},L_{\vartheta_T(X)})$ is a Bachmann-Howard system. We can now invoke Definition~\ref{def:collapse-bh-system-inequality} an Lemma~\ref{lem:collapse-linear} to equip $\vartheta^2_T(X)$ with a linear order. To show that $\vartheta_T(X)$ is good we establish the implication
 \begin{equation*}
  s<_{\vartheta_T(X)}t\quad\Rightarrow\quad\iota_{\vartheta_T(X)}(s)<_{\vartheta^2_T(X)}\iota_{\vartheta_T(X)}(t),
 \end{equation*}
 by induction on $L_{\vartheta_T(X)}(s)+L_{\vartheta_T(X)}(t)$. First assume that $s=\vartheta\sigma<_{\vartheta_T(X)}\vartheta\tau=t$ holds by clause~(i) of Definition~\ref{def:collapse-bh-system-inequality}. This means that we have
 \begin{equation*}
  \sigma<_{T_X}\tau\qquad\text{and}\qquad[\iota_X]^{<\omega}(\supp^T_X(\sigma))\lef_{\vartheta_T(X)}t.
 \end{equation*}
 Clearly we have $T_{\iota_X}(\sigma)<_{T_{\vartheta_T(X)}}T_{\iota_X}(\tau)$. To conclude
 \begin{equation*}
  \iota_{\vartheta_T(X)}(s)=\vartheta\, T_{\iota_X}(\sigma)<_{\vartheta^2_T(X)}\vartheta\, T_{\iota_X}(\tau)=\iota_{\vartheta_T(X)}(t)
 \end{equation*}
 we need to establish
 \begin{equation*}
  [\iota_{\vartheta_T(X)}]^{<\omega}(\supp^T_{\vartheta_T(X)}(T_{\iota_X}(\sigma)))\lef_{\vartheta^2_T(X)}\iota_{\vartheta_T(X)}(t).
 \end{equation*}
 For any $r\in\supp^T_{\vartheta_T(X)}(T_{\iota_X}(\sigma))=[\iota_X]^{<\omega}(\supp^T_X(\sigma))$ we have $r<_{\vartheta_T(X)}t$ by assumption. The induction hypothesis yields $\iota_{\vartheta_T(X)}(r)<_{\vartheta^2_T(X)}\iota_{\vartheta_T(X)}(t)$, as required. A similar argument applies if $s<_{\vartheta_T(X)}t$ holds by clause~(ii) of Definition~\ref{def:collapse-bh-system-inequality}.
\end{proof}

By Definition~\ref{def:collapse-approximations} we obtain a collapse $\vartheta_{\vartheta_T(X)}:T_{\vartheta_T(X)}\rightarrow\vartheta_T(\vartheta_T(X))$. The following shows that the diagram from the introduction commutes:

\begin{proposition}\label{prop:collapses-commute}
 Assume that $X$ and thus $\vartheta_T(X)$ is a good Bachmann-Howard system for a prae-dilator $T$. Then we have $\iota_{\vartheta_T(X)}\circ\vartheta_X=\vartheta_{\vartheta_T(X)}\circ T_{\iota_X}$.
\end{proposition}
\begin{proof}
 Unravelling definitions we compute
 \begin{equation*}
  \vartheta_{\vartheta_T(X)}(T_{\iota_X}(\sigma))=\vartheta\, T_{\iota_X}(\sigma)=\iota_{\vartheta_T(X)}(\vartheta\sigma)=\iota_{\vartheta_T(X)}(\vartheta_X(\sigma)),
 \end{equation*}
 as promised.
\end{proof}

\section{The minimal Bachmann-Howard fixed point}

In the previous section we have given a detailed construction of the diagram from the introduction. The goal of this section is to investigate its direct limit. We have already observed that the empty order $\emptyset$, together with the unique functions $\iota_\emptyset:\emptyset\rightarrow \vartheta_T(\emptyset)$ and $L_\emptyset:\emptyset\rightarrow\omega$, is a good Bachmann-Howard system for any prae-dilator. Together with Theorem~\ref{thm:bh-systems-iterate} we can construct the following objects:

\begin{definition}\label{def:bhT}
 Consider a prae-dilator $T$. We build a sequence of good Bachmann-Howard systems by setting
 \begin{align*}
  (X_0,\iota_{X_0},L_{X_0})&:=(\emptyset,\iota_\emptyset,L_\emptyset),\\
  (X_{n+1},\iota_{X_{n+1}},L_{X_{n+1}})&:=(\vartheta_T(X_n),\iota_{\vartheta_T(X_n)},L_{\vartheta_T(X_n)}).
 \end{align*}
 Define the order $\bh(T)$ as the direct limit of the system $(X_n,\iota_{X_n}:X_n\rightarrow X_{n+1})_{n\in\omega}$. It comes with embeddings $j_{X_n}:X_n\rightarrow\bh(T)$ that satisfy $j_{X_{n+1}}\circ \iota_{X_n}=j_{X_n}$.
\end{definition}

As explained in the introduction, the present paper is supposed to be formalized in primitive recursive set theory ($\prs$). Let us briefly discuss the formalization of the above constructions (more details can be found in \cite[Sections~1.1,~1.2,~2.2]{freund-thesis}): Given a primitive recursive family $(T^u)_{u\in\mathbb V}$ of prae-dilators, it is straightforward to see that the transformation $(u,X)\mapsto\vartheta_{T^u}(X)$ is a primitive recursive set function, and that the properties from the previous section can be established in~$\prs$. Write $X_n^u$ for the Bachmann-Howard systems from the above definition, constructed with respect to $T^u$. Invoking primitive recursion along the ordinals we see that $(u,n)\mapsto X^u_n$ is a primitive recursive set function. It follows that the transformation of $u$ into the underlying set of the direct limit $\bh(T^u)$ is primitive recursive as well, since the latter can be explicitly represented by
\begin{equation*}
 \bh(T^u)=\{(n,s)\,|\,s\in X^u_{n+1}\land s\notin\rng(\iota_{X^u_n})\}.
\end{equation*}
Similarly, one checks that the universal property is witnessed by a primitive recursive transformation (see~\cite[Lemma~2.2.17]{freund-thesis}). In particular we can use the universal property (in the category of sets) to construct the limit order on $\bh(T^u)$. Thus we finally learn that $u\mapsto(\bh(T^u),<_{\bh(T^u)})$ is a primitive recursive set function. Let us now come to the first of our main results:

\begin{theorem}\label{thm:bhT-fixed point}
 For each prae-dilator $T$, the order $\bh(T)$ is a Bachmann-Howard fixed point of $T$.
\end{theorem}
\begin{proof}
 In order to construct a Bachmann-Howard collapse $\vartheta:T_{\bh(T)}\rightarrow\bh(T)$ we will exploit the fact that $T_{\bh(T)}$ is a direct limit of the system
 \begin{equation*}
 (T_{X_n},T_{\iota_{X_n}}:T_{X_n}\rightarrow T_{X_{n+1}})_{n\in\omega}.
 \end{equation*}
Indeed, Girard's original definition explicitly demands that (prae-)dilators preserve direct limits. Since we have worked with a different formulation of the definition we shall give a short proof of this fact: Consider an arbitrary $\sigma\in T_{\bh(T)}$. Since the support $\supp^T_{\bh(T)}(\sigma)$ is a finite subset of $\bh(T)$ it is contained in the range of some embedding $j_{X_n}$. Using clause~(ii) of Definition~\ref{def:prae-dilator} we can infer that $\sigma$ lies in the range of $T_{j_{X_n}}$. Thus we have established
\begin{equation*}
 T_{\bh(T)}=\bigcup_{n\in\omega}\rng(T_{j_{X_n}}),
\end{equation*}
which ensures that $T_{\bh(T)}$, together with the functions $T_{j_{X_n}}:T_{X_n}\rightarrow T_{\bh(T)}$, is the desired direct limit (both in the category of linear orders and in the category of sets). Relying on Definition~\ref{def:collapse-approximations}, let us now consider the functions
 \begin{equation*}
  j_{X_{n+1}}\circ\vartheta_{X_n}:T_{X_n}\rightarrow\bh(T).
 \end{equation*}
 We can use Proposition~\ref{prop:collapses-commute} to compute
 \begin{multline*}
  (j_{X_{n+2}}\circ\vartheta_{X_{n+1}})\circ T_{\iota_{X_n}}=j_{X_{n+2}}\circ(\vartheta_{\vartheta_T(X_n)}\circ T_{\iota_{X_n}})=\\
  =j_{X_{n+2}}\circ(\iota_{\vartheta_T(X_n)}\circ\vartheta_{X_n})=(j_{X_{n+2}}\circ\iota_{X_{n+1}})\circ\vartheta_{X_n}=j_{X_{n+1}}\circ\vartheta_{X_n}.
 \end{multline*}
 Now the universal property of $T_{\bh(T)}$ yields a function
 \begin{equation*}
  \vartheta:T_{\bh(T)}\rightarrow\bh(T)\qquad\text{with}\qquad\vartheta\circ T_{j_{X_n}}=j_{X_{n+1}}\circ\vartheta_{X_n}.
 \end{equation*}
 We have to verify the conditions from Definition~\ref{def:bachmann-howard-collapse}: Aiming at condition~(i), consider elements $\sigma,\tau\in T_{\bh(T)}$ with
 \begin{equation*}
  \sigma<_{T_{\bh(T)}}\tau\qquad\text{and}\qquad\supp^T_{\bh(T)}(\sigma)\lef_{\bh(T)}\vartheta(\tau).
 \end{equation*}
 Pick $n$ large enough to write $\sigma=T_{j_{X_n}}(\sigma_0)$ and $\tau=T_{j_{X_n}}(\tau_0)$ with $\sigma_0,\tau_0\in T_{X_n}$. Then we have $\sigma_0<_{T_{X_n}}\tau_0$, as well as
 \begin{multline*}
  [j_{X_{n+1}}]^{<\omega}\circ[\iota_{X_n}]^{<\omega}(\supp^T_{X_n}(\sigma_0))=[j_{X_n}]^{<\omega}(\supp^T_{X_n}(\sigma_0))=\\
  =\supp^T_{\bh(T)}(T_{j_{X_n}}(\sigma_0))\lef_{\bh(T)}\vartheta(\tau)=\vartheta\circ T_{j_{X_n}}(\tau_0)=j_{X_{n+1}}\circ\vartheta_{X_n}(\tau_0).
 \end{multline*}
 Since $j_{X_{n+1}}$ is an order embedding we obtain $[\iota_{X_n}]^{<\omega}(\supp^T_{X_n}(\sigma_0))\lef_{X_{n+1}}\vartheta_{X_n}(\tau_0)$. Now Proposition~\ref{prop:collapses-admissible} yields $\vartheta_{X_n}(\sigma_0)<_{X_{n+1}}\vartheta_{X_n}(\tau_0)$ and then
 \begin{equation*}
  \vartheta(\sigma)=j_{X_{n+1}}\circ\vartheta_{X_n}(\sigma_0)<_{\bh(T)}j_{X_{n+1}}\circ\vartheta_{X_n}(\tau_0)=\vartheta(\tau).
 \end{equation*}
 To establish condition~(ii) of Definition~\ref{def:bachmann-howard-collapse} we again write $\sigma=T_{j_{X_n}}(\sigma_0)$. By Proposition~\ref{prop:collapses-admissible} we have $[\iota_{X_n}]^{<\omega}(\supp^T_{X_n}(\sigma_0))\lef_{X_{n+1}}\vartheta_{X_n}(\sigma_0)$. This implies
 \begin{multline*}
  \supp^T_{\bh(T)}(\sigma)=\supp^T_{\bh(T)}(T_{j_{X_n}}(\sigma_0))=[j_{X_n}]^{<\omega}(\supp^T_{X_n}(\sigma_0))=\\
  =[j_{X_{n+1}}]^{<\omega}\circ[i_{X_n}]^{<\omega}(\supp^T_{X_n}(\sigma_0))\lef_{\bh(T)}j_{X_{n+1}}\circ\vartheta_{X_n}(\sigma_0)=\vartheta\circ T_{j_{X_n}}(\sigma_0)=\vartheta(\sigma),
 \end{multline*}
 just as required.
\end{proof}

The previous results were formulated for arbitrary prae-dilators, whether or not they preserve well-foundedness. Restricting our attention to dilators, we obtain a more explicit version of the Bachmann-Howard principle:

\begin{definition}
 The predicative Bachmann-Howard principle is the assertion that $\bh(T)$ is well-founded for any dilator $T$.
\end{definition}

The nomenclature alludes to the view that the construction of $\bh(T)$ is predicatively acceptable, since it is realized by a primitive recursive set function (cf.~\cite{feferman-predicative-set}). To avoid misunderstanding we point out that the well-foundedness of $\bh(T)$ cannot be established by predicative means: Indeed, we will see that the predicative Bachmann-Howard principle is equivalent to $\Pi^1_1$-comprehension. This equivalence also ensures that the predicative Bachmann-Howard principle is sound, which is not trivial at all (in general, well-foundedness is not preserved under direct limits). Theorem~\ref{thm:bhT-fixed point} shows that the predicative Bachmann-Howard principle implies its abstract counterpart. The converse implication follows from the fact that $\bh(T)$ is the minimal Bachmann-Howard fixed point:

\begin{theorem}\label{thm:bhT-minimal}
 Consider a prae-dilator $T$. The order $\bh(T)$ can be embedded into any Bachmann-Howard fixed point of $T$.
\end{theorem}
\begin{proof}
 Let $Y$ be a Bachmann-Howard fixed point of $T$, witnessed by a Bachmann-Howard collapse $\vartheta_Y:T_Y\rightarrow Y$. Given a good Bachmann-Howard system $(X,\iota_X,L_X)$ and an embedding $h_X:X\rightarrow Y$, we can define a function $h_X^\vartheta:\vartheta_T(X)\rightarrow Y$ by
 \begin{equation*}
  h_X^\vartheta(\vartheta\sigma)=\vartheta_Y(T_{h_X}(\sigma)).
 \end{equation*}
 If we have
 \begin{equation*}
  h_X^\vartheta\circ\iota_X=h_X,
 \end{equation*}
 then $h_X$ is called a Bachmann-Howard interpretation of $X$. The main step of the present proof is to establish the following claim: If $h_X$ is a Bachmann-Howard interpretation of $X$, then $h_X^\vartheta$ is a Bachmann-Howard interpretation of $\vartheta_T(X)$ (in particular it is an order embedding). Based on this claim we can conclude as follows: Clearly the empty map $h_{X_0}:X_0=\emptyset\rightarrow Y$ is a Bachmann-Howard interpretation of $X_0$. Iteratively we can then construct Bachmann-Howard interpretations
 \begin{equation*}
h_{X_{n+1}}:=h_{X_n}^\vartheta:X_{n+1}\rightarrow Y  
 \end{equation*}
of the orders from Definition~\ref{def:bhT}. The definition of Bachmann-Howard interpretation ensures $h_{X_{n+1}}\circ\iota_{X_n}=h_{X_n}$. Thus the universal property of the limit~$\bh(T)$ allows us to glue the embeddings $h_{X_n}:X_n\rightarrow Y$ to the desired embedding of $\bh(T)$ into~$Y$. In order to establish the open claim we consider a Bachmann-Howard interpretation $h_X:X\rightarrow Y$. The implication
 \begin{equation*}
  s<_{\vartheta_T(X)}t\rightarrow h_X^\vartheta(s)<_Y h_X^\vartheta(t)
 \end{equation*}
 can be established by induction on $L_{\vartheta_T(X)}(s)+L_{\vartheta_T(X)}(t)$. Let us first consider the case where $s=\vartheta\sigma<_{\vartheta_T(X)}\vartheta\tau=t$ holds because of
 \begin{equation*}
  \sigma<_{T_X}\tau\qquad\text{and}\qquad [\iota_X]^{<\omega}(\supp^T_X(\sigma))\lef_{\vartheta_T(X)}t.
 \end{equation*}
Then we get $T_{h_X}(\sigma)<_{T_Y} T_{h_X}(\tau)$. Also recall that $L_{\vartheta_T(X)}(\iota_X(x))<L_{\vartheta_T(X)}(s)$ holds for any $x\in\supp^T_X(\sigma)$. Thus the definition of Bachmann-Howard interpretation and the induction hypothesis yield
 \begin{multline*}
  \supp^T_Y(T_{h_X}(\sigma))=[h_X]^{<\omega}(\supp^T_X(\sigma))=\\
  =[h_X^\vartheta]^{<\omega}\circ[\iota_X]^{<\omega}(\supp^T_X(\sigma))\lef_Y=h_X^\vartheta(t)=\vartheta_Y(T_{h_X}(\tau)).
 \end{multline*}
 In view of Definition~\ref{def:bachmann-howard-collapse} we obtain the desired inequality
 \begin{equation*}
  h_X^\vartheta(s)=\vartheta_Y(T_{h_X}(\sigma))<_Y \vartheta_Y(T_{h_X}(\tau))=h_X^\vartheta(t).
 \end{equation*}
 Next, assume that $s=\vartheta\sigma<_{\vartheta_T(X)}\vartheta\tau=t$ holds because of $\tau<_{T_X}\sigma$ (which is in fact redundant) and $s\leq^{\operatorname{fin}}_{\vartheta_T(X)}[\iota_X]^{<\omega}(\supp^T_X(\tau))$. Parallel to the above we obtain
 \begin{equation*}
  h_X^\vartheta(s)\leq^{\operatorname{fin}}_Y\supp^T_Y(T_{h_X}(\tau)).
 \end{equation*}
Since Definition~\ref{def:bachmann-howard-collapse} provides the inequality
\begin{equation*}
 \supp^T_Y(T_{h_X}(\tau))\lef_Y\vartheta_Y(T_{h_X}(\tau))=h_X^\vartheta(t),
\end{equation*}
we can infer $h_X^\vartheta(s)<_Yh_X^\vartheta(t)$ by transitivity. So far we have established that $h_X^\vartheta:\vartheta_T(X)\rightarrow Y$ is an embedding. To conclude that it is a Bachmann-Howard interpretation we consider the function $h^{\vartheta^2}_X:=(h_X^\vartheta)^\vartheta:\vartheta_T(\vartheta_T(X))\rightarrow Y$ and compute
 \begin{equation*}
  h^{\vartheta^2}_X\circ\iota_{\vartheta_T(X)}(\vartheta\sigma)=h^{\vartheta^2}_X(\vartheta\, T_{\iota_X}(\sigma))=\vartheta_Y(T_{h_X^\vartheta}\circ T_{\iota_X}(\sigma))=\vartheta_Y(T_{h_X}(\sigma))=h_X^\vartheta(\vartheta\sigma),
 \end{equation*}
 using the assumption that $h_X$ is a Bachmann-Howard interpretation.
\end{proof}

We can now complete the proof of Theorem~\ref{thm:pred-bhp}, which was stated in the introduction:

\begin{proof}
 In view of Theorem~\ref{thm:abstract-bhp} (which was established in~\cite{freund-equivalence}, based on similar results in~\cite{freund-bh-preprint,freund-thesis}) it suffices to show that the abstract and the predicative Bachmann-Howard principle are equivalent. To show that the former implies the latter we assume that $Y$ is a well-founded Bachmann-Howard fixed point of a given dilator $T$. By the previous theorem there is an order embedding of $\bh(T)$ into $Y$. This ensures that $\bh(T)$ is well-founded as well, as demanded by the predicative Bachmann-Howard principle. For the other direction we consider a dilator $T$ and assume that $\bh(T)$ is well-founded. From Theorem~\ref{thm:bhT-fixed point} we know that $\bh(T)$ is a Bachmann-Howard fixed point of $T$. Thus $\bh(T)$ itself serves as a witness for the abstract Bachmann-Howard principle.
\end{proof}

As explained in the introduction, the point of the predicative Bachmann-Howard principle is that it separates the construction of a Bachmann-Howard fixed point from the question of well-foundedness. Thus it splits the impredicative principle of $\Pi^1_1$-comprehension into a predicative construction and a statement about the preservation of well-foundedness.

\bibliographystyle{amsplain}
\bibliography{Bibliography_Freund}

\end{document}